\documentclass[11pt]{amsart}

\usepackage[a4paper,hmargin=3cm,vmargin=3cm]{geometry}
\usepackage{amsfonts,amssymb,amscd,amstext}
\usepackage{graphicx}
\usepackage[dvips]{epsfig}
\usepackage[latin1]{inputenc}

\usepackage{fancyhdr}
\usepackage{times}

\usepackage{enumerate}
\usepackage{titlesec}
\usepackage{mathrsfs}
\usepackage{stmaryrd}

\def\R{\mathbb{R}}

\def\L{\mathbb{L}}

\def\M{\mathbb{M}}

\newcommand{\ben}{\begin{enumerate}}
\newcommand{\bit}{\begin{itemize}}
\newcommand{\een}{\end{enumerate}}
\newcommand{\eit}{\end{itemize}}

\newcommand{\ed}{\end{document}}

\def\cA{\mathcal{A}}
\def\cU{\mathcal{U}}

\def\cW{\mathcal{W}}

\def\cZ{\mathcal{Z}}

\def\cL{\mathcal{L}}
\def\cM{\mathcal{M}}

\def\cF{\mathcal{F}}

\def\cK{\mathcal{K}}

\let\8=\infty \let\0=\emptyset  
\let\hat=\widehat

\let\landa=\lambda
\let\alfa=\alpha

\let\parc=\partial

\def\ep{\varepsilon}

\def\landa{\lambda}

\def\flecha{\rightarrow}
\def\esiz{\langle}
\def\esde{\rangle}

\def\S{\Sigma}

\def\cte.{\mathop{\rm cte.}\nolimits}

\def\R{\mathbb{R}}

\def\H{\mathbb{H}}
\def\S{\mathbb{S}}

\newfont{\bb}{msbm10 at 12pt}

\titleformat{\section}
{\filcenter\bfseries\large} {\thesection{.}}{0.2cm}{}
\titleformat{\subsection}[runin]
{\bfseries} {\thesubsection{.}}{0.15cm}{}[.]
\titleformat{\subsubsection}[runin]
{\em}{\thesubsubsection{.}}{0.15cm}{}[.]

\usepackage[up,bf]{caption}

\newtheorem{theorem}{Theorem}[section]
\newtheorem{lemma}[theorem]{Lemma}

\newtheorem{remark}[theorem]{Remark}
\newtheorem{corollary}[theorem]{Corollary}

\newtheorem{assertion}[theorem]{Assertion}

\theoremstyle{definition}

\numberwithin{equation}{section}
\numberwithin{figure}{section}

\usepackage{color}

\begin{document}
\fancyhead[LO]{Analytic saddle spheres in $\S^3$ are equatorial}
\fancyhead[RE]{José A. Gálvez, Pablo Mira, Marcos P. Tassi}
\fancyhead[RO,LE]{\thepage}

\thispagestyle{empty}

\begin{center}
{\bf \LARGE Analytic saddle spheres in $\S^3$ are equatorial}
\vspace*{5mm}

\hspace{0.2cm} {\Large José A. Gálvez, Pablo Mira, Marcos P. Tassi}
\end{center}

\footnote[0]{
\noindent \emph{Mathematics Subject Classification}: 53A10, 53C42 \\ \mbox{} \hspace{0.25cm} \emph{Keywords}: Topological index, saddle surfaces, umbilic points, minimal surfaces, immersed spheres.}



\vspace*{7mm}

\begin{quote}
{\small
\noindent {\bf Abstract.}\hspace*{0.1cm}
A theorem by Almgren establishes that any minimal $2$-sphere immersed in $\S^3$ is a totally geodesic equator. In this paper we give a purely geometric extension of Almgren's result, by showing that any immersed, real analytic $2$-sphere in $\S^3$ that is saddle, i.e., of non-positive extrinsic curvature, must be an equator of $\S^3$. We remark that, contrary to Almgren's theorem, no geometric PDE is imposed on the surface. The result is not true for $C^{\8}$ spheres.


\vspace*{0.1cm}

}
\end{quote}


\section{Introduction}

A theorem of Almgren \cite{A}, sometimes also called the \emph{Calabi-Almgren theorem} after \cite{Ca}, establishes that any minimal $2$-sphere $\Sigma$ immersed in the unit round sphere $\S^3$ must be \emph{equatorial}, i.e., a totally geodesic $2$-sphere of $\S^3$. In this paper we show that the Calabi-Almgren theorem is a particular case of a much more general geometric result, in which one does not ask the surface to be minimal or to satisfy any other geometric PDE. We will merely impose that $\Sigma$ be diffeomorphic to $\S^2$, real analytic and \emph{saddle}, i.e., $\kappa_1\kappa_2\leq 0$ at each point of $\Sigma$, where $\kappa_1,\kappa_2$ are the principal curvatures of $\Sigma$. Note that this is indeed a wide generalization of the Calabi-Almgren theorem, since any minimal surface in a real analytic Riemannian $3$-manifold $(M^3,g)$ is trivially saddle and real analytic.

\begin{theorem}\label{th:main1}
Any immersed, real analytic saddle sphere in $\S^3$ is equatorial.
\end{theorem}

It is important to remark that there exist non-equatorial $C^{\8}$-\emph{smooth} saddle spheres in $\S^3$; see Section \ref{sec:examples}. Thus, the real analyticity condition of Theorem \ref{th:main1} is necessary and sharp. Moreover, since there exist compact minimal surfaces in $\S^3$ of arbitrary genus \cite{L}, the topological hypothesis in Theorem \ref{th:main1} is also necessary.  

We also remark that the saddle condition $\kappa_1\kappa_2\leq 0$ for a surface $\Sigma$ in $\S^3$ is actually a purely geometric one, since it is equivalent to imposing locally on $\Sigma$ the \emph{convex hull property}, i.e., that for any sufficiently small neighborhood $\cW\subset \Sigma$ of any point $p\in \Sigma$, the set $\cW$ is contained in the convex hull in $\S^3$ of the boundary curve $\parc \cW$.

The proof of Theorem \ref{th:main1} relies on controlling the umbilic set  $\cU$ of any non-equatorial saddle sphere $\Sigma$ in the conditions of Theorem \ref{th:main1}, and on the description around $\cU$ of the \emph{cross field} on $\Sigma$ determined by its principal directions. We will show that, even though the umbilic set $\cU$ can contain in principle a number of real analytic arcs, maybe crossing at some \emph{meeting points}, the cross field generated  by the principal directions of the surface on $\Sigma\setminus \cU$ extends analytically to define two orthogonal real analytic line fields $\cF_1,\cF_2$ with a finite number of singularities $\{q_1,\dots, q_k\}$ on $\Sigma$. Then, we will prove that the topological index of any of these extended line fields $\cF_j$ is non-positive at each singularity $q_i$. Since $\Sigma$ has genus zero, this will contradict the Poincaré-Hopf theorem applied to $\cF_j$, showing that $\Sigma$ must be equatorial.

Theorem \ref{th:main1} is related to a number of uniqueness problems of classical surface theory, and also to more recent advances. Some of them will be discussed in Section \ref{sec:discussion}. In Section \ref{sec:examples} we will produce a simple example of a $C^{\8}$ saddle sphere immersed in $\S^3$. 
In Section \ref{sec:analytic} we will prove an analytic result that controls, up to a rotation in the $(x,y)$-coordinates, the index of the line field generated by the vector field $\nabla h_x$ around the origin for any real analytic function $h(x,y)$ satisfying the saddle condition $h_{xx}h_{yy}-h_{xy}^2\leq 0$. In Section \ref{sec:index}, this result will be used to describe, for any real analytic saddle surface $\Sigma$ in $\S^3$, the analytic extension and the topological index around any umbilic point $p$ of $\Sigma$ of the cross field generated by the principal directions of $\Sigma$ around $p$. Theorem \ref{th:main1} will follow from this description. 
\section{Discussion of the result}\label{sec:discussion}

The Calabi-Almgren theorem is related to the classification problem of minimal $2$-spheres inside Riemannian $3$-spheres $(\S^3,g)$. By the Smith-Simon theorem in \cite{Sm}, the space of minimal $2$-spheres immersed in some Riemannian $(\S^3,g)$ is always non-empty, but the actual classification of such minimal $2$-spheres is a hard problem. Regarding uniqueness, one has as a particular case of the uniqueness theory developed by the first two authors in \cite{GM3} that if a family of minimal $2$-spheres $\cM:=\{\Sigma_{\alfa}\}_{\alfa}$ in some $(\S^3,g)$ has the property that their generalized Gauss maps foliate the unit tangent bundle $TU(\S^3)$ of $\S^3$, then \emph{any} minimal $2$-sphere immersed in $(\S^3,g)$ is an element of this canonical family $\cM$. In a recent paper, Ambrozio, Marques and Neves \cite{AMN} called such a family $\cM$ a \emph{Zoll family} of minimal surfaces, and proved the existence of non-homogeneous Riemannian $3$-spheres $(\S^3,g)$ for which a Zoll family $\cM$ exists, and thus the uniqueness theorem in \cite{GM3} applies.

Note that when $g$ is the round metric on $\S^3$, the family of totally geodesic equators $\S^2\subset \S^3$ constitutes a Zoll family on $(\S^3,g)$. More generally, we have:

\begin{enumerate}
\item
Any homogeneous metric $g$ on $\S^3$ has the property that, in suitable coordinates, all the equators $\S^2\subset \S^3$ are minimal $2$-spheres in $(\S^3,g)$; they are totally geodesic only when $g$ is the round metric. See \cite{To2} for the case of Berger $3$-spheres, and \cite{AMN} for the general case.
\item
There exist non-homogeneous metrics $g$ on $\S^3$ such that all the equators $\S^2\subset \S^3$ are minimal $2$-spheres in $(\S^3,g)$. These metrics were classified in \cite{AMN}.
\item
For any metric $g$ in $\S^3$ with minimal equators, any minimal $2$-sphere immersed in $(\S^3,g)$ is an equator, by \cite{GM3} and item (1). See also \cite{AMN}, and \cite{mmpr0} for the homogeneous case.
\end{enumerate}

\emph{Theorem \ref{th:main1} gives a wide generalization of assertion (3) above.} Indeed, if $(\S^3,g)$ has the property that all the equators of $\S^3$ are minimal, it follows immediately from the maximum principle applied to the minimal surface equation (for $g$) that any minimal $2$-sphere $\Sigma$ in $(\S^3,g)$ is also saddle when viewed as a surface in the round sphere $\S^3$. Moreover, by the classification in \cite{AMN} of metrics $g$ on $\S^3$ with minimal equators, we obtain that $g$, and so $\Sigma$, is real analytic. Hence, by Theorem \ref{th:main1}, $\Sigma$ must be an equator of $\S^2$.

Almgren conceived his uniqueness theorem in \cite{A} as the natural spherical version in $\S^3$ of Bernstein's theorem according to which entire minimal graphs in $\R^3$ are planes. Similarly, Theorem \ref{th:main1} gives the natural spherical version in $\S^3$ of another related, famous theorem by Bernstein which establishes that any entire, bounded graph $z=z(x,y)$ in $\R^3$ that is saddle, i.e. $z_{xx}z_{yy}-z_{xy}^2\leq 0$, must be flat.

Alexandrov was the first to study the extension of this \emph{second Bernstein theorem} to $\S^3$. In \cite{A0}, he showed that any real analytic ovaloid in $\R^3$ whose principal curvatures satisfy the Weingarten inequality $(\kappa_1-c)(\kappa_2-c)\leq 0$ for some $c>0$ must be a round sphere of radius $1/c$. Using a projective equivalence, this implies that any real analytic saddle sphere in $\S^3\subset \R^4$ that is an entire graph over $\S^2$ must be an equator; here by an \emph{entire graph} we mean that $\Sigma$ intersects transversally all geodesic arcs of $\S^3$ joining the north and south poles. Alexandrov conjectured in \cite{Ale} that the analyticity hypothesis could be removed from this theorem. This conjecture was also formulated later on by Koutrofiotis and Nirenberg \cite{K}. In 2001, Martinez-Maure \cite{MM} found a striking, beautiful $C^2$ counterexample to this conjecture. Based on this construction, Panina found later on $C^{\8}$ counterexamples \cite{Pa1}. The umbilic set of the examples by Martinez-Maure is the union of four disjoint great semicircles. In \cite{GM6}, the first two authors showed the converse statement, i.e., that any $C^2$ ovaloid in $\R^3$ that satisfies the Weingarten inequality must be totally umbilic along four disjoint great semicircles. For other uniqueness theorems of ovaloids, or immersed spheres in $\R^3$, satisfying the Weingarten inequality, see \cite{GM3,GMT,GWZ,HNY,Mu}.

\section{Smooth saddle spheres in $\S^3$}\label{sec:examples}

Let $(x_0,x_1,x_2,x_3)$ be Euclidean coordinates in $\R^4$, and view $\S^3\subset \R^4$ as the unit sphere. Then, it is classically known that the map 
\begin{equation}\label{deffi}
\varphi (x_0,x_1,x_2,x_3)=\frac{1}{x_0}(x_1,x_2,x_3):\S_+^3\flecha \R^3
\end{equation} 
defines a totally geodesic diffeomorphism from the open hemisphere $\S_+^3 :=\S^3\cap \{x_0>0\}$ into the Euclidean space $\R^3$. That is, $\varphi$ takes geodesics of $\S_+^3$ into geodesics of $\R^3$. From this property, one sees directly that if $\Sigma$ is an immersed surface in $\S_+^3$ and $p\in \Sigma$, then the extrinsic curvature $K_e:=\kappa_1 \kappa_2$ of $\Sigma$ at $p$ has the same sign as the Gauss curvature of $\varphi (\Sigma)$ in $\R^3$ at $\varphi (p)$. In particular, $\varphi$ takes saddle surfaces of $\S_+^3$ into saddle surfaces of $\R^3$, and vice versa.

The same construction also works for the hyperbolic $3$-space $\H^3$ when viewed in its hyperboloid model $\H^3=\{x\in \L^4 : \esiz x,x\esde=-1, x_0\geq 1\}$ of Minkowski $4$-space $\L^4\equiv (\R^4,\esiz,\esde)$, where $\esiz,\esde=-dx_0^2+\sum_{i=1}^3 dx_i^2$. In this case, the map \eqref{deffi} gives a totally geodesic diffeomorphism between $\H^3$ and the unit ball of $\R^3$.

\begin{figure}[htbp]
        \includegraphics[width=7.8cm]{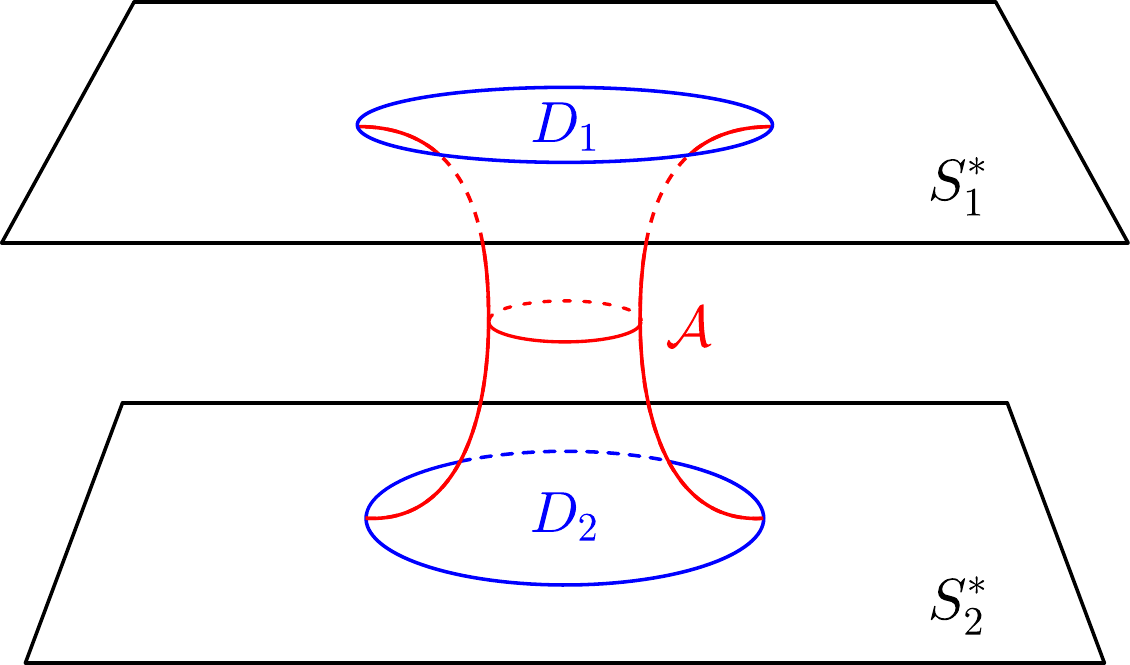}\hspace{1cm} \includegraphics[width=5.7cm]{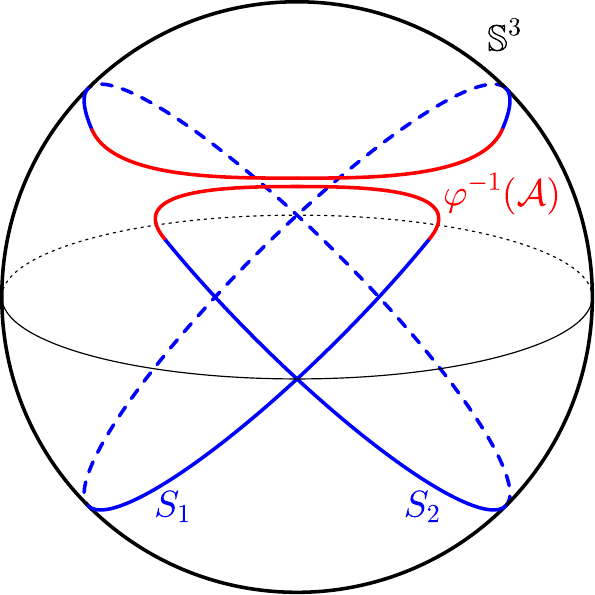}
     \caption{Left: the saddle annulus $\cA$ in $\R^3$ joining $S_1^*\setminus D_1$ and $S_2^*\setminus D_2$. Right: schematic view of the immersed saddle sphere $\Sigma$ in $\S^3$. } 
\label{fig:sphere}
\end{figure}

We will use the map \eqref{deffi} to construct a simple saddle sphere immersed in $\S^3$. To start, consider the equatorial $2$-spheres in $\S^3$ given by $S_1=\S^3\cap \{x_0=x_3\}$ and $S_2=\S^3\cap \{x_0=-x_3\}$. Then, by \eqref{deffi}, $$S_1^*:= \varphi(S_1\cap \S_3^+),\hspace{0.5cm} S_2^*:= \varphi(S_2\cap \S_3^+)$$ are, respectively, the planes $z=1$ and $z=-1$ in $\R^3$, where here $(x,y,z)$ are the usual Euclidean coordinates of $\R^3$. We next remove from $S_i^*$, $i=1,2$, the disk $D_i$ given by the points that satisfy $x^2+y^2<1$, and glue together $C^{\8}$-smoothly the remaining exterior domains $S_1^*\setminus D_1$ and $ S_2^*\setminus D_2$ through an embedded saddle annulus $\cA\subset \R^3$ such that $\parc \cA=\parc D_1\cup \parc D_2$. See Figure \ref{fig:sphere}.

Note that each $S_i\setminus \varphi^{-1}(D_i)$ is homeomorphic to a disk. Then, the union in $\S^3$ of $S_1\setminus \varphi^{-1}(D_1)$ and $S_2\setminus \varphi^{-1}(D_2)$ through the annulus $\varphi^{-1}(\cA)$ defines a $C^{\8}$-smooth immersed sphere $\Sigma$ in $\S^3$. Since $\varphi$ preserves saddleness and $S_1,S_2$ are totally geodesic in $\S^3$, it is clear that $\Sigma$ is a $C^{\8}$ saddle sphere in $\S^3$.

The same process using a finite number of compact saddle annuli $\cA_1,\dots, \cA_{g+1}$ in $\R^3$ clearly creates simple examples of compact saddle surfaces in $\S^3$ of any genus $g$.

\section{Analytic saddle functions and topological index}\label{sec:analytic}

\begin{lemma}\label{lem:recta}
Let $h=h(x,y)$ be a nonlinear real analytic function that satisfies $h_{xx}h_{yy}-h_{xy}^2 \leq 0$. Assume that the Hessian matrix $D^2 h$ vanishes along a curve $\Gamma\subset \R^2$. Then, $\Gamma$ is a line segment.
\end{lemma}
\begin{proof}
Since $D^2 h$ is identically zero along $\Gamma$, we have that $\nabla h\equiv (a,b)\in \R^2$ along $\Gamma$. Therefore, after changing $h$ by $h(x,y)-ax-by-c$ for an adequate $c\in \R$ (note that this function has the same Hessian as $h$), we can assume that $(a,b)=(0,0)$ and that $h\equiv 0$ along $\Gamma$.

We want to show that $\Gamma$ is a line segment. Arguing by contradiction, assume that  $\Gamma$ has non-zero curvature at some $p\in \Gamma$. Since $h$ is real analytic and not identically zero, the nodal set $h^{-1}(0)$ is, around $p$, a finite union of real analytic planar curves passing through $p$. Thus, by choosing a different nearby point in $\Gamma$ as $p$ if necessary, we can assume that the set $h^{-1}(0)$ around $p$ is just a piece of the curve $\Gamma$.

Up to a translation and a rotation in the $(x,y)$-coordinates, we can assume that $p=(0,0)$, and so $h(0,0)=0$, and that $\Gamma$ is tangent to the $x$-axis at $(0,0)$. Moreover, since $\Gamma$ has non-zero curvature at $p$, we can also assume that for some sufficiently small $r_0>0$, the set $\Gamma\cap D(0;r_0)$ lies in the half-plane $\{y\leq 0\}$, and only intersects the $y=0$ axis at the origin.

\begin{figure}[htbp]
         \includegraphics[width=10cm]{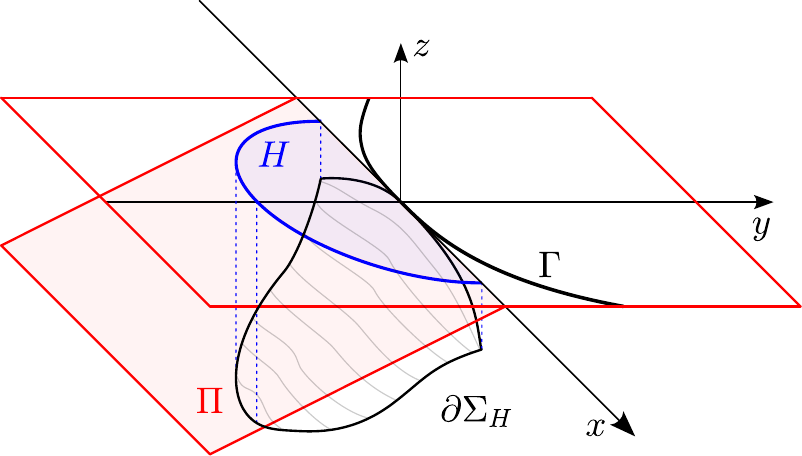} 
     \caption{The saddle graph $\Sigma_H$ and the plane $\Pi$ in the argument of Lemma \ref{lem:recta}.} 
\label{fig:saddle}
\end{figure}

Consider next the half-disk $H:= \{(x,y): x^2+y^2 \leq r_0^2, y\leq 0\}$ and the compact saddle graph with boundary $$\Sigma_H:=\{(x,y,h(x,y)): (x,y)\in H\}\subset \R^3.$$ Note that the tangent plane to $\Sigma_H$ at the origin is the $z=0$ plane. Also, $h(x,y)\neq 0$ for any point in the interior of $H$. Assume for definiteness that $h(x,y)<0$ in ${\rm int}(H)$.  Let $m_0<0$ denote the maximum value of $h$ over the closed half-circle $S$ contained in $\parc H$.  Then, $\parc \Sigma_H$ lies below the closed half-space in $\R^3$ determined by the plane $\Pi$ of equation $z=\frac{-m_0}{r_0}y$. See Figure \ref{fig:saddle}. Note that $\Pi$ contains the $x$-axis, and lies below the plane $z=0$ for negative values of $y$. Since $\Sigma_H$ is saddle, it follows from the convex hull property for saddle surfaces that $\Sigma_H$ must lie in the convex hull of its boundary values $\parc \Sigma_H$: in particular, $\Sigma_H$ must lie below $\Pi$. But this contradicts that $z=0$ is the tangent plane to $\Sigma_H$ at the origin. This completes the proof.
\end{proof}

In what follows we will use the notation $h_{\nu}:=\esiz \nabla h,\nu\esde$, where $\nu\in \R^2$. We recall that a \emph{line field} on a surface $\Sigma$ continuously assigns to each $p\in \Sigma$ a line $L_p\subset T_p\Sigma$ passing through the origin. In particular, a line field $L$ defined on a punctured neighborhood of a point $q\in \Sigma$ has a topological index at $q$, which is a half-integer that measures the total variation of $L$ along any small positively oriented Jordan curve in $\Sigma$ surrounding $q$. See Hopf \cite{Ho}.
\begin{theorem}\label{th:analytic}
Let $h(x,y)$ be a real analytic, nonlinear function defined on a neighborhood of $(0,0)$, that satisfies $h_{xx}h_{yy}-h_{xy}^2 \leq 0$. Then:
\begin{enumerate}
\item
There exists a direction $\nu=(\cos \theta,\sin \theta)\in \S^1$ and a (possibly empty) finite collection of line segments $\Gamma_j$ passing through the origin such that $\nabla h_{\nu} \neq (0,0)$ on $\Omega^*\setminus \cup_{j} \Gamma_j$, and $D^2 h$ vanishes identically along each $\Gamma_j$. Here, $\Omega$ is a sufficiently small disk centered at the origin, and $\Omega^*:=\Omega\setminus \{(0,0)\}$
\item
The line field generated by $\nabla h_{\nu}$ in $\Omega^*\setminus \cup_j \Gamma_j$ extends to a real analytic line field $\cF$ defined in $\Omega^*$, that is orthogonal to each segment $\Gamma_j$.
 \item
The topological index of the line field $\cF$ at $(0,0)$ is non-positive.
\end{enumerate}
\end{theorem}
\begin{remark}\label{rem:nu}
The proof will actually show that the statement of Theorem \ref{th:analytic} is true for any $\nu\in \S^1$ except for a finite number of directions $\nu_1,\dots, \nu_s\in \S^1$. The radius of the related disk $\Omega=\Omega(\nu)$ depends on the choice of $\nu\in \S^1\setminus \{\nu_1,\dots, \nu_s\}$. 
\end{remark}
\begin{proof}
Along the proof we will assume that $\Omega$ is a sufficiently small disk around the origin, of radius $r_0>0$. Also without loss of generality we will assume that $h(0,0)=0$ and $\nabla h (0,0)=(0,0)$, changing $h$ by $h -ax-by-c$ as in Lemma \ref{lem:recta} if necessary.

To start, we consider an arbitrary direction $\nu\in \S^1$. Then, $\nabla h_{\nu}$ defines a real analytic line field $\cL_{\nu}$ in $\Omega\setminus \cZ_{\nu}$, where $\cZ_{\nu}$ is the set of critical points of $h_{\nu}$ in $\Omega$. Taking a smaller $\Omega$ if necessary, we have by the analytic implicit function theorem that either $\cZ_{\nu}=\{(0,0)\}$, or $\cZ_{\nu}=\Omega$, or $\cZ_{\nu}$ is the union of a finite set of regular, embedded analytic curves that pass through $(0,0)$ and intersect only at the origin. Obviously, any  \emph{umbilical segment} $\Gamma$ along which the Hessian matrix of $h$ vanishes (see Lemma \ref{lem:recta}) is contained in $\cZ_{\nu}$. In these conditions, we have:
\begin{lemma}\label{lem:ext}
Let $\Gamma$ be an umbilical segment of $h$. Then, for every direction $\nu\in \S^1$ transverse to $\Gamma$, it holds $\cZ_{\nu}\neq \Omega$, and the line field $\cL_{\nu}$ generated by $\nabla h_{\nu}$ extends analytically across $\Gamma\setminus \{(0,0)\}$, so that $\cL_{\nu}(p)$ is orthogonal to $\Gamma$ at each $p\in \Gamma\setminus \{(0,0)\}$.
\end{lemma}
\begin{proof}
We may assume without loss of generality that $\Gamma$ is a segment of the line $x=0$. Since $\Gamma$ is umbilical and $h$ is not identically zero, we can then write 
\begin{equation}\label{humbilic}
h(x,y)= x^n w(x,y),
\end{equation}
where $n>2$ and $w(x,y)$ is real analytic, with $w(0,y)\neq 0$ for $y\neq 0$ small enough. If we write $\nu=(\cos \theta, \sin \theta)$, a calculation from \eqref{humbilic} shows that 
\begin{equation}\label{fraca0}
\nabla h_{\nu} (x,y)= x^{n-2} \{ \left(n(n-1) \cos \theta w(x,y), 0\right) + x ( \cdots) \}.
\end{equation}
In particular, since $\cos\theta\neq 0$, $\nabla h_{\nu}$ does not vanish identically, i.e. $\cZ_{\nu}\neq \Omega$. Thus, it follows from \eqref{fraca0} that for $x\neq 0$ and  $y\neq 0$ small enough so that $w(0,y)\neq 0$, we have 
\begin{equation}\label{fraca}
\frac{\nabla h_{\nu} (x,y)}{|\nabla h_{\nu} (x,y)|} \to (\pm 1,0)  \text{ as $(x,y)\to (0,y)$ }
\end{equation}
with either $x>0$ or $x<0$. 
Here, the sign $\pm$ only depends on the signs of $\cos \theta$, of $w(0,y)$ and of $x^{n-2}$. In particular, if $n$ is odd, we have a sign for $x>0$ and another one for $x<0$. In any case, \eqref{fraca} shows that the line field $\cL_{\nu}$ extends analytically across $\Gamma\setminus \{(0,0)\}$, while being orthogonal to $\Gamma$ at any $p\in \Gamma\setminus \{(0,0)\}$. This proves Lemma \ref{lem:ext}.
\end{proof}

\begin{lemma}\label{lem:flat}
Theorem \ref{th:analytic} holds if $h_{xx}h_{yy}-h_{xy}^2\equiv 0$.
\end{lemma} 
\begin{proof}
Let $\omega(x,y)$ be the homogeneous polynomial of lowest degree $n\geq 2$ in the series expansion of $h(x,y)$ at the origin. It is then clear that $\omega_{xx} \omega_{yy}-\omega_{xy}^2=0$. So, $z=\omega(x,y)$ is a complete flat graph in $\R^3$. By the well known classification of such graphs due to Pogorelov \cite{Pog}, see also \cite{HN}, we deduce that $\omega(x,y)$ depends only on one variable, i.e., $\omega(x,y)=f(\alfa x + \beta y)$, for adequate constants $\alfa,\beta$ and a certain function $f(v)$. Since $\omega$ is a homogeneous polynomial, we have $\omega(x,y)=a(\alfa x+ \beta y)^n$ for some $a\neq 0$ and $n\geq 2$.

If $n=2$, then $D^2 h$ is different from zero at the origin, and in that case the result is immediate since either $\nabla h_{x}$ or $\nabla h_y$ is non-zero at the origin. The topological index of $\cF$ at $(0,0)$ in this case is trivially zero.

Assume next $n>2$. After a rotation in the $(x,y)$ coordinates, we can then write 
\begin{equation}\label{hflat}
h(x,y)=\hat{a} x^n + \cdots, \hspace{0.5cm} \hat{a}\neq 0.
\end{equation}
The graph $\Sigma$ given by $z=h(x,y)$ is flat and real analytic in $\R^3$. It is then classically known that $\Sigma$ is foliated by straight lines, which are principal directions for the null principal curvature $\kappa=0$ of $\Sigma$. It is also classically known that if $\gamma\subset \Sigma$ is any such straight line, then either all points of $\gamma$ are umbilical (i.e., the second fundamental form of $\Sigma$ vanishes along $\gamma$), or $\gamma$ has no umbilical points. See Section 5.8 in \cite{doc}.

In the language of Theorem \ref{th:analytic}, and since we are assuming that $n>2$ and so $D^2 h$ vanishes at $(0,0)$, we deduce that there exists a unique line segment $\Gamma$ passing through the origin such that $D^2 h$ vanishes identically along $\Gamma$, and $D^2h$ has rank one in $\Omega\setminus \Gamma$. By \eqref{hflat}, $\Gamma$ is contained in the $x=0$ axis. In particular, \eqref{humbilic} holds, and \eqref{hflat} implies that $w(0,0)\neq 0$ in \eqref{humbilic}. Thus, the argument in Lemma \ref{lem:ext} shows that for any $\nu\in \S^1$ that is transverse to $\Gamma$, the line field $\cL_{\nu}$ extends analytically across the origin, and in particular has index zero at $(0,0)$. This proves Lemma \ref{lem:flat}.
\end{proof}

\emph{In the rest of the proof we will assume that, by Lemma \ref{lem:flat},} $h_{xx}h_{yy}-h_{xy}^2\not\equiv 0.$%

\begin{lemma}\label{lem:ext2}
There exists a direction $\nu\in \S^1$ such that the line field $\cL_{\nu}$ extends analytically to $\Omega^*$. In particular, it has a well defined topological index at the origin.
\end{lemma}
\begin{proof}
Let $\cK\subset \Omega$ be the set of points where $h_{xx}h_{yy}-h_{xy}^2=0$. Since $h_{xx}h_{yy}-h_{xy}^2\not\equiv 0$, we have by the analytic implicit function theorem that either $\cK\cap \Omega^*$ is empty, or $\cK$ is the union of a finite number of regular, embedded, real analytic curves that intersect only at the origin. Note that the umbilical segments $\Gamma_j$, in case they exist, are trivially among such curves of $\cK$. See Figure \ref{fig:queso1}. Making $\Omega$ smaller, we can assume that $D^2 h$ has rank one (except maybe at the origin) along any curve $\beta_j$ of $\cK$ that is not an umbilical segment. 

\begin{figure}[htbp]
      \includegraphics[width=4.7cm]{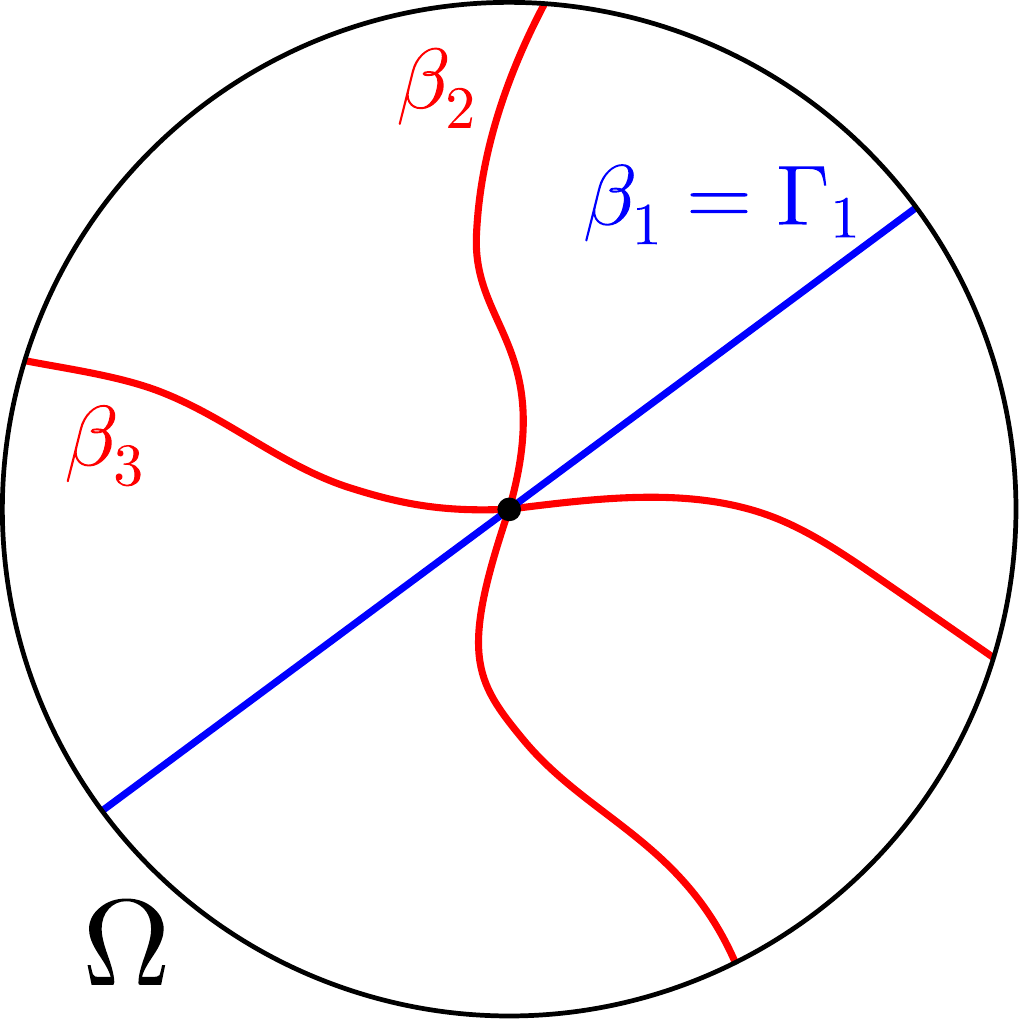} 
     \caption{The curves $\beta_j$, including the umbilical segments $\Gamma_j$, of $\cK$.} 
\label{fig:queso1}
\end{figure}

Let $\nu\in \S^1$ be an arbitrary direction. By analyticity, and making $\Omega$ smaller if necessary, we have by ${\rm det}(D^2 h)\not\equiv 0$ that the set of critical points of $h_{\nu}$ is either empty, or the origin, or a finite union of regular, real analytic curves passing through $(0,0)$. Moreover, along any curve over which $\nabla h_{\nu}\equiv (0,0)$ holds, one has ${\rm det}(D^2 h)=0$, i.e., the curve is one of the curves $\beta_j$ that compose $\cK$. On the other hand, if $\beta_j$ is not an umbilic segment $\Gamma_j$, then for any direction $\nu^0\in \S^1$ linearly independent from $\nu$, it holds $\nabla h_{\nu^0}(p)\neq (0,0)$ at any $p\in \beta_j\cap \Omega^*$, since $D^2 h$ has rank one at every point of $\beta_j\cap \Omega^*$.

In particular, there is at most a finite number of directions $\nu=\nu_j\in \S^1$ for which $\nabla h_{\nu}\equiv (0,0)$ holds along some non-umbilical curve $\beta_j\in \cK$. Moreover, by the above discussion, any $\nu\in \S^1$ linearly independent to these directions $\nu_j$ satisfies that $\nabla h_{\nu} \neq (0,0)$ in $\Omega^*\setminus \cup_j \Gamma_j$. Choose now such $\nu$ so that, additionally, is transverse to all the umbilical segments $\Gamma_j$. Then, the line field $\cL_{\nu}$ generated by $\nabla h_{\nu}$, which is at first only defined in $\Omega^*\setminus \cup_j \Gamma_j$, extends analytically to $\Omega^*$, by Lemma \ref{lem:ext}, and is orthogonal to each $\Gamma_j$. This proves Lemma \ref{lem:ext2}.
\end{proof}

Observe that, once here, we have already proved the first two assertions of Theorem \ref{th:analytic}. So, to finish, we now need to show that the topological index of the line field $\cF:= \cL_{\nu}$ is $\leq 0$ at the origin. We will do it in two steps:  (a) when there are no umbilical segments $\Gamma_j$, i.e., if $D^2 h$ is never zero in a punctured neighborhood $\Omega^*$ of the origin, and (b) if there exist umbilical segments $\Gamma_j$.

We start with the first case.

\begin{lemma}\label{lem:iso}
Assume that $D^2 h$ does not vanish at any point of the punctured disk $\Omega^*$. Then, the topological index of $\cF$ at $(0,0)$ is $\leq 0$.
\end{lemma}
\begin{proof}
Assume without loss of generality, after a rotation in the $(x,y)$-coordinates, that $\nu =(1,0)$, i.e., $h_{\nu}=h_x$. Arguing by contradiction, assume that $\cF$ has positive index at the origin. This means that the vector field $\nabla h_x$ has positive index at the origin (note that $(0,0)$ is an isolated critical point of $h_x$ in the present case).

In that situation, it follows that $h_x$ has a strict local maximum or minimum at $(0,0)$ (see Lemma 3.1 in \cite{AM}), i.e. either $h_x>0$ or $h_x<0$ in $\Omega^*$, maybe choosing $\Omega$ smaller if necessary. For definiteness, let us assume that $h_x>0$ in $\Omega^*$. Then, the mapping $$\sigma =(h_x,h_y):\Omega\flecha \R^2$$ has its image contained in the \emph{extended} half-plane $\{(p,q): p>0\}\cup \{(0,0)\}$.

Since $h_{xx}h_{yy}-h_{xy}^2 \leq 0$, we have ${\rm det}(J\sigma) \leq 0$, where $J\sigma$ denotes the Jacobian matrix of $\sigma$. Thus, we are in the conditions of  Hartman-Nirenberg \cite{HN}, which yields that $$\parc (\sigma(\Omega)) \subseteq \sigma (\parc \Omega).$$ That is, the boundary of the image $\sigma(\Omega)$ is contained in the image of the boundary $\parc \Omega$.

In our situation, $\sigma (\parc\Omega)$ is contained in the open half-plane $\{(p,q): p>0\}$, and so it does not pass through $(0,0)$. Since $\sigma(0,0)=(0,0)$ and $\sigma(\Omega)\subset \{(p,q): p>0\}\cup \{(0,0)\}$, we have $(0,0)\in \parc (\sigma(\Omega))$. This contradicts the result by Hartman and Nirenberg, what proves Lemma \ref{lem:iso}.
\end{proof}

We finally study the case in which there exist umbilical segments $\Gamma_j$.

\begin{lemma}\label{lem:noniso}
Assume that $D^2 h$ does not vanish at any point in $\Omega^*\setminus \{\Gamma_1,\dots, \Gamma_k\}$, where each $\Gamma_j$ is a segment passing through the origin along which $D^2h$ vanishes. Then, the topological index of $\cF$ at $(0,0)$ is $\leq 0$.
\end{lemma}
\begin{proof}
Assume, up to a rotation in the $(x,y)$-coordinates, that $\nu=(1,0)$ for the direction $\nu\in \S^1$ in Lemma \ref{lem:ext2}. Thus, $\nabla h_{\nu}=\nabla h_x$. Then, as shown in the proof of Lemma \ref{lem:ext2}, we have that $\nabla h_x \neq (0,0)$ in $\Omega^*\setminus \{\Gamma_1,\dots,\Gamma_k\}$. By analyticity, the nodal set $h_x^{-1}(0)$ of $h_x$ in $\Omega$ is the union of (see Figure \ref{fig:queso2})

\begin{enumerate}
\item
A finite set of regular, embedded analytic curves $\gamma_1,\dots, \gamma_s$ passing through the origin, and along which $\nabla h_x\neq (0,0)$ except at $(0,0)$ (note that $\nabla h_x$ is normal to any such curve), and
\item
The umbilical segments $\Gamma_1,\dots, \Gamma_k$.
\end{enumerate}

\begin{figure}[htbp]
      \includegraphics[width=5.1cm]{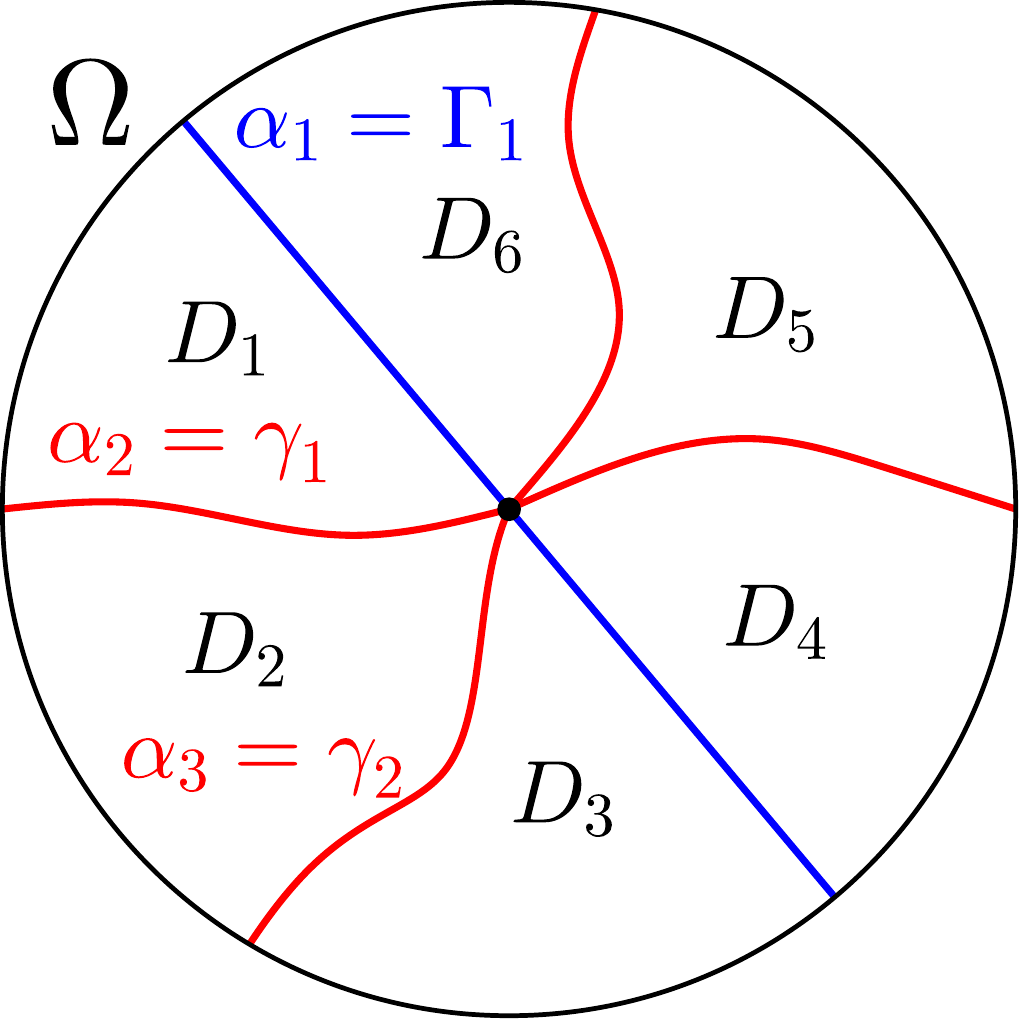} 
     \caption{The curves $\alfa_j$ that compose the nodal set $h_x^{-1}(0)$, and the associated nodal regions.} 
\label{fig:queso2}
\end{figure}

By Lemma \ref{lem:ext}, the line field $\cF$ generated by $\nabla h_x$ extends to $\Omega^*$, and $\cF(p)$ is orthogonal to $\Gamma_j$ at each $p\in \Gamma_j\cap \Omega^*$.

Let us denote by $\alfa_1,\dots, \alfa_r$ to the union of the curves $\gamma_j$ and $\Gamma_j$, see Figure \ref{fig:queso2}. Clearly, the curves $\alfa_j$ divide $\Omega^*$ into an even number of disjoint open \emph{sectors} $D_1,\dots, D_{2r}$. For each such sector $D_j$, we will let $s_{1,j},s_{2,j}$ denote the two endpoints of the circular arc $\parc\Omega\cap \parc D_j$, ordered with respect to the positive orientation of $\parc\Omega$.

Observe that $h_x$ has a sign on each $D_j$. Moreover, since $\nabla h_x\neq (0,0)$ in $\cup_j D_j$, each $D_j$ is foliated by regular, analytic level curves of $h_x$, with normal directions given by $\nabla h_x$. 

We will next define a Jordan curve $\Upsilon\subset \overline{\Omega}$ that encloses $(0,0)$, in order to compute later on the total variation of the line field $\cF$ along $\Upsilon$.

First, for each $D_j$, the intersection $\Upsilon\cap D_j$ is defined by choosing some non-empty level curve $h_{x}^{-1}(\ep_j)\cap D_j$. Clearly, we can take $\ep_j$ small enough so that the level curve $h_{x}^{-1}(\ep_j)\cap D_j$ is connected. We denote then by $t_{1,j},t_{2,j}$ the two intersection points of this curve with $\parc \Omega$, again ordered with respect to the positive orientation of $\parc \Omega$; see Figure \ref{fig:queso3}, left. Second, we add to these level curves the circular arcs in $\parc D_j\cap \parc \Omega$ that joint $s_{1,j}$ with $t_{1,j}$ and $s_{2,j}$ with $t_{2,j}$. The union of these elements define a piecewise analytic Jordan curve $\Upsilon$ that encloses the origin. See Figure \ref{fig:queso3}, right.

\begin{figure}[htbp]\vspace{0.3cm}
      \includegraphics[width=5.2cm]{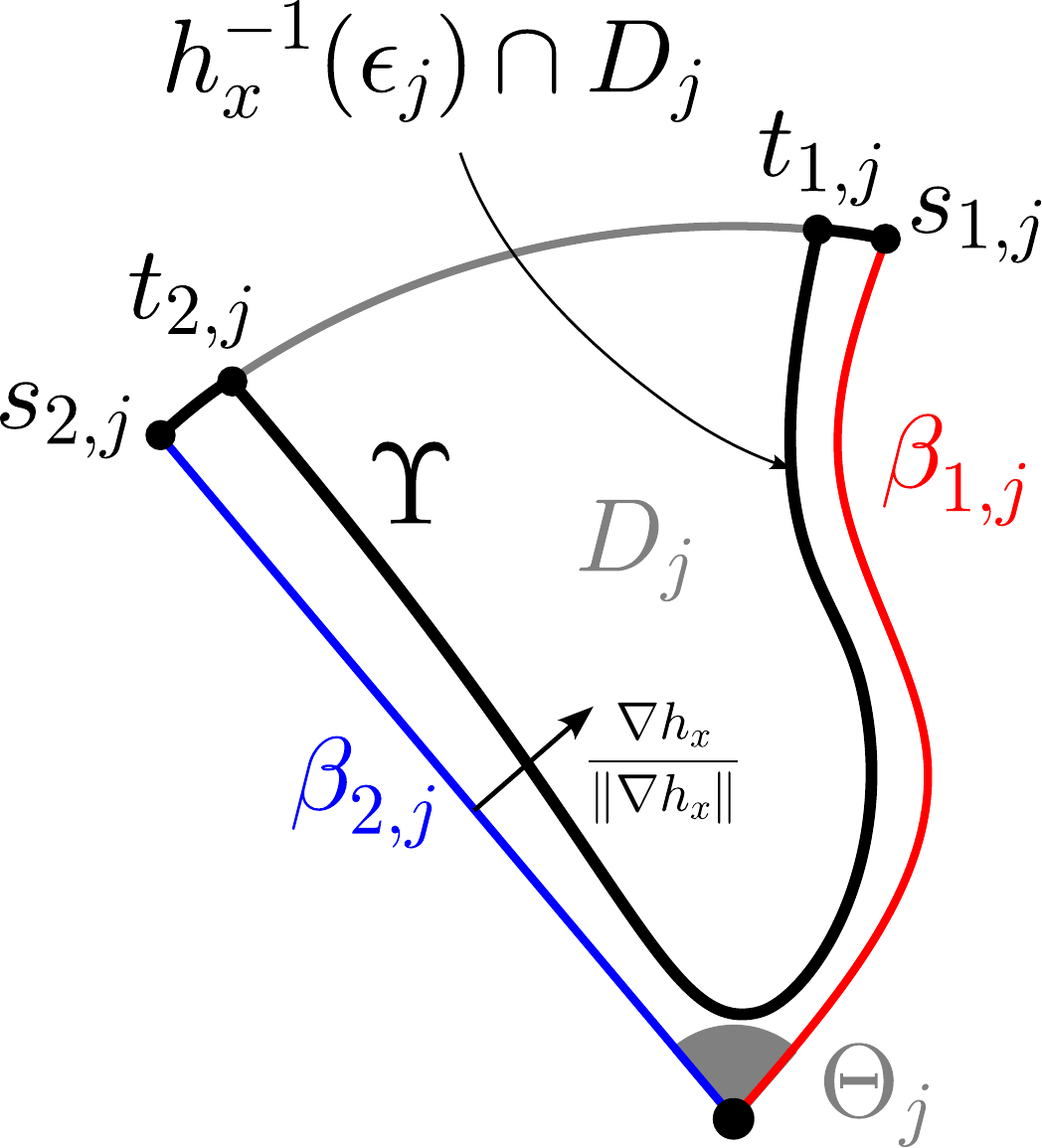} \hspace{1.3cm} \includegraphics[width=5.5cm]{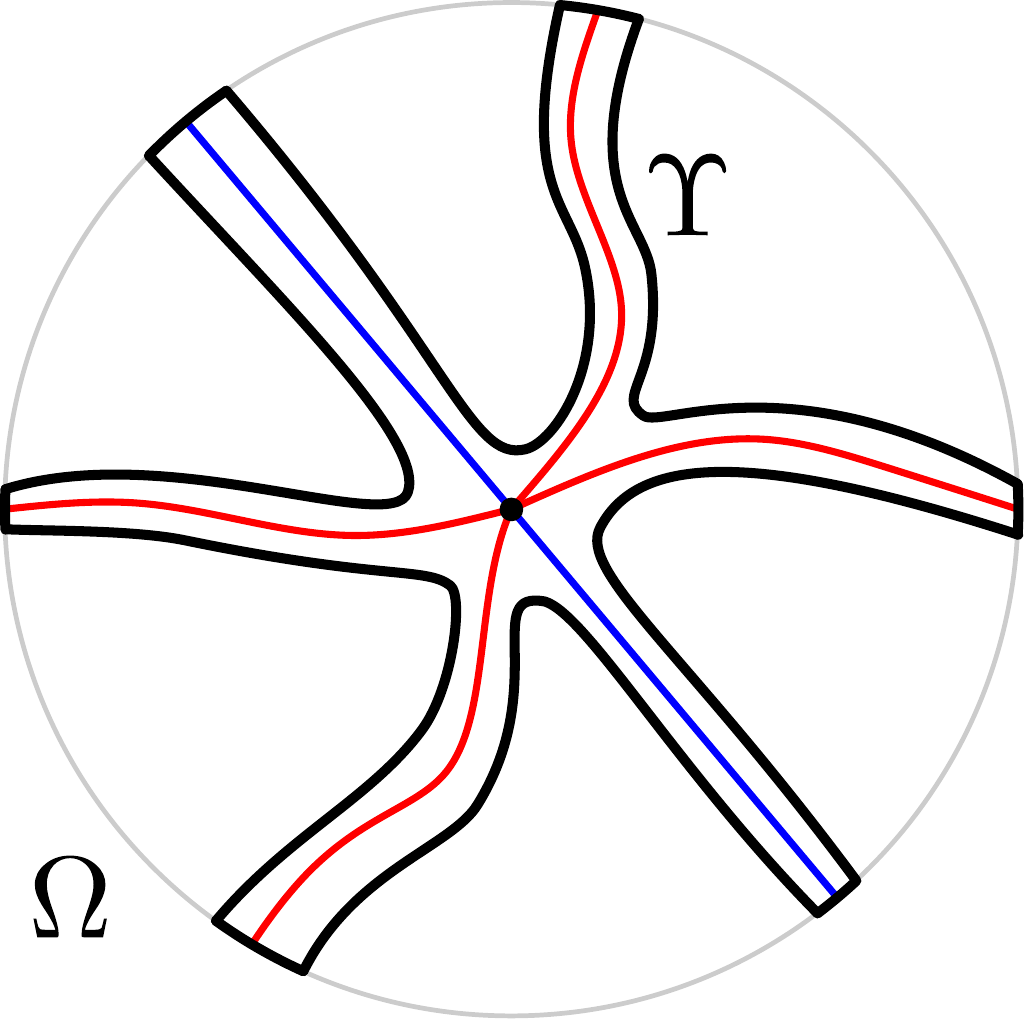}
     \caption{The definition of the curve $\Upsilon$.} 
\label{fig:queso3}
\end{figure}

Note that the variation of the vector field $\nabla h_x$ along each level arc $h_x^{-1}(\ep_j)\cap D_j$ of $\Upsilon$ coincides with $\theta_{2,j}-\theta_{1,j}$, where $\theta_{i,j}$ is the angle that $\nabla h_x$ makes with the positive $x$-axis at $t_{i,j}$. Now, when $\ep_j\to 0$, the set $h_x^{-1}(\ep_j)\cap D_j$ converges to the union of the two arcs $\beta_{1,j},\beta_{2,j}$ in $\parc D_j$ that join the origin with the vertices $s_{1,j}$, $s_{2,j}$. Moreover, around any $p\in \beta_{i,j}\cap \Omega^*$, $i\in \{1,2\}$, the unit normal of $h_x^{-1}(\ep_j)\cap D_j$ that points in the $\nabla h_x$ direction converges as $\ep_j\to 0$ to the interior (resp. exterior) unit normal of $\parc D_j$ at $p$ if $h_x>0$ (resp. $h_x<0$) in $D_j$. Indeed, this is immediate if $\beta_{i,j}$ lies in one of the nodal curves $\gamma_1,\dots,\gamma_s$, and it follows from Lemma \ref{lem:ext} and its proof if $\beta_{i,j}$ is a piece of an umbilical segment $\Gamma_1,\dots, \Gamma_k$.

In this way, when both $\ep_j$ and the radius $r_0$ of $\Omega$ converge to zero, the variation of $\nabla h_x$ along $h_x^{-1}(\ep_j)\cap D_j$ converges to $\Theta_j-\pi  \leq 0$, where $\Theta_j\in [0,\pi]$ is the angle that the analytic arcs $\beta_{1,j},\beta_{2,j}$ make at the origin. 

Therefore, choosing $r_0$ and ${\rm max} (|\ep_j|: j\in \{1,\dots, 2r\})$ small enough, we can assume that the variation of $\nabla h_x$ (or equivalently, of the line field $\cF$) along $h_x^{-1}(\ep_j)\cap D_j$ is smaller than $\Theta_j-\pi + \pi/(8 r)$.

Besides, for such small $r_0$ fixed, the lengths of the circular arcs of $\Upsilon$ contained in $\parc \Omega$ converge to zero as ${\rm max}_j\{|\ep_j|\}\to 0$. Since at any $p\in\alfa_j\cap \Omega^*$ the line field $\cF$ is orthogonal to $\alfa_j$, we deduce that the variation of $\cF$ along each of the $2r$ circular arcs of $\Upsilon$ is smaller than $\pi/(8r)$, choosing ${\rm max}_j\{|\ep_j|\}$ small enough. In this way, for the resulting Jordan curve $\Upsilon =\Upsilon (\ep_j,r_0)$, the total variation $\delta(\cF)$ of 
$\cF$ along $\Upsilon$ is smaller than $2(1-r)\pi + \pi/2\leq \pi/2$: for this, note that $\sum_{j=1}^{2r} \Theta_j = 2\pi$. Now, the quantity $\delta(\cF)/(2\pi)$ gives the topological index ${\rm Ind}(\cF)$ of $\cF$ at $(0,0)$, that must be a half-integer. Thus, since $\delta(\cF)< \pi/2$, we conclude that ${\rm Ind}(\cF)\leq 0$. This proves Lemma \ref{lem:noniso} and Theorem \ref{th:analytic}.
\end{proof}
\end{proof}

\section{Topological index of principal line fields}\label{sec:index}
Let $\mathbb{M}^3$ denote a space form $\R^3, \H^3$ or $\S^3$. Then, for an immersed surface $\Sigma$ in $\M^3$, its principal line fields $\cL_1,\cL_2$ associated to the principal curvatures $\kappa_1\geq \kappa_2$ are well defined and smooth in $\Sigma\setminus \cU$, where $\cU$ is the set of umbilic points of $\Sigma$.

\begin{theorem}\label{th:indice}
Let $\Sigma$ be an immersed real analytic surface in $\M^3$, with $\kappa_1\kappa_2\leq 0$ at each point, and let $q_0\in \cU\subset \Sigma$. Assume that $\Sigma$ is not totally geodesic. Then, there exist two orthogonal real analytic line fields $\cF_1,\cF_2$ defined on a punctured neighborhood $D^*\subset \Sigma$ of $q_0$ with the following properties:
\begin{enumerate}
\item
At any non-umbilic point $p\in D^*$, we have $\cF_1\cup \cF_2=\cL_1\cup \cL_2$. That is, $\cF_1,\cF_2$ point at the principal directions of $\Sigma$ at $p$.
 \item
The topological index at $q_0$ of $\cF_1,\cF_2$ is $\leq 0$.
\end{enumerate}
\end{theorem}
\begin{proof}
We will prove the result for $\M=\S^3$, i.e., the case in which we are most interested. The case $\M=\H^3$ is proved similarly, with an adequate change of sign, using the totally geodesic model for $\H^3$ (see Section \ref{sec:examples}), while the case $\M=\R^3$ follows from the same argument, but simplified.

After an isometry of $\S^3$, we assume $q_0=(1,0,0,0)$. Then, using the totally geodesic model of $\S_+^3$ described in Section \ref{sec:examples}, there exist adequate local coordinates $(x,y)$ in a neighborhood $\Omega\subset \R^2$ of $(0,0)$ such that $\Sigma$ can be parametrized locally around $q_0$ as $\psi(x,y):\Omega\subset \R^2\flecha \S^3$, $$\psi(x,y)= \frac{1}{\sqrt{1+x^2+y^2+h(x,y)^2}}(1,x,y,h(x,y)),$$ where $h(x,y)$ is a real analytic function in $\Omega$ with $h(0,0)=0$ and $Dh(0,0)=(0,0)$; note that in this parametrization $\psi(0,0)=q_0$. We will also assume that $\Sigma$ is not totally geodesic. In the rest of the argument, we will always consider that $\Omega$ is a disk centered at $(0,0)$, as small as necessary.

Denoting $p=h_x$, $q=h_y$, $r=h_{xx}$, $s=h_{xy}$, $t=h_{yy}$, the first fundamental form $I$ of $\Sigma$ in these $(x,y)$-coordinates is given by 
\begin{equation}\label{1ff}
\left(\begin{array}{cc}E & F\\ F& G \end{array} \right)=\frac{1}{(1+x^2+y^2+h^2)^2} \, \cM,
\end{equation} 
where
$$
\cM:=\left(
\def\arraystretch{1.5}\begin{array}{cc}
(h-p x)^2+\left(1+p^2\right) \left(1+y^2\right)  & p \left(q \left(1+x^2+y^2\right)-h y\right)-x (h q+y) \\
 p \left(q \left(1+x^2+y^2\right)-h y\right)-x (h q+y) & (h-q y)^2+\left(1+q^2\right) \left(1+x^2\right) \\
\end{array}
\right).$$
Similarly, the coefficients of the second fundamental form $II$ of $\Sigma$ with respect to $(x,y)$ are
\begin{equation}\label{2ff}
\left(\begin{array}{cc}e & f\\ f & g \end{array} \right) = \frac{1}{\sqrt{1+x^2+y^2+h^2}\sqrt{1+p^2+q^2+(x p+y q -h)^2}}\left(\begin{array}{cc}r & s\\ s& t \end{array} \right).
\end{equation}
In particular, as $\kappa_1\kappa_2\leq 0$, we have $h_{xx}h_{yy}-h_{xy}^2\leq 0$. Note that the umbilical points of $\Sigma$ near $q_0$ correspond to the points $(x,y)\in \Omega$ where $D^2 h$ vanishes. Since $\Sigma$ is not totally geodesic, it follows from Lemma \ref{lem:recta} and the analyticity of $\Sigma$ that either $q_0$ is an isolated umbilical point of $\Sigma$, and in this case $D^2 h$ is never zero in the punctured disk $\Omega^*$, or else the set of umbilical points around $q_0$ corresponds to the union of a finite number of segments $\Gamma_1,\dots, \Gamma_k\subset \Omega$ passing through the origin. In particular, this time $D^2 h$ vanishes only on $\cup_j \Gamma_j$.

The principal line fields $\cL_1,\cL_2$ in the $(x,y)$-coordinates are given by the solutions to
 \begin{equation}\label{rectas1}
 -\alfa_{12}dx^2 + (\alfa_{11}-\alfa_{22}) dx dy + \alfa_{21} dy^2=0,
 \end{equation}
where $(\alfa_{ij}):=II\cdot I^{-1}$, with $I,II$ being the first and second fundamental forms \eqref{1ff}, \eqref{2ff} of $\Sigma$. The line fields given by \eqref{rectas1} are well defined at first only in $\Omega^*\setminus \cup_j \Gamma_j$. We show next that they admit analytic extensions to $\Omega^*$.

\begin{assertion}\label{ass:1}
There exist two real analytic line fields $\cF_1,\cF_2$ in $\Omega^*$ that solve \eqref{rectas1} at every point of $\Omega^*$. Moreover, for every $p\in \Gamma_j\cap \Omega^*$, one of $\cF_1(p),\cF_2(p)$ is tangent to $\Gamma_j$ at $p$, and the other one is normal.
\end{assertion}
\begin{proof}
If there are no umbilical segments $\Gamma_j$, the result is immediate, taking $\cF_i=\cL_i$ for $i=1,2$. So, from now on, we assume that there exists at least one umbilical segment $\Gamma$ that passes through $(0,0)$. Up to a rotation in the $(x,y)$ coordinates, we will assume that such segment is contained in the $x=0$ axis. So, 
\begin{equation}\label{achen}
h(x,y)=x^n \eta(x,y)
\end{equation} 
for some $n>2$ and some real analytic function $\eta(x,y)$ with $\eta(0,y)\neq 0$ for $y\neq 0$ sufficiently small. We remark that the number $n$ measures the vanishing order of $h$ along $\Gamma$, and is independent of the performed rotation in the $(x,y)$ coordinates.

A direct computation using \eqref{1ff}, \eqref{2ff} and \eqref{achen} shows that $\alfa_{ij}= x^{n-2} \hat\alfa_{ij}$, where
\begin{equation}\label{eq:ago}
\left(\begin{array}{cc}\hat\alfa_{11} & \hat \alfa_{12}\\ \hat\alfa_{21}& \hat\alfa_{22} \end{array} \right) = \left(\begin{array}{cc}n(n-1)\eta(0,y)(1+y^2)^{-2} & 0\\ 0 & 0 \end{array} \right) + x \left(\begin{array}{cc} * & *\\ * & * \end{array} \right).
\end{equation}

As $\eta(0,y)\neq 0$ for $y\neq 0$ small enough, we deduce that the equation
 \begin{equation}\label{rectas8}
 -\hat\alfa_{12} dx^2 + (\hat \alfa_{11}-\hat \alfa_{22}) dx dy + \hat\alfa_{21} dy^2=0,
 \end{equation}
for $\hat \alfa_{ij}$ in \eqref{eq:ago}, defines two real analytic line fields $\cF_a,\cF_b$ around every point of the form $(0,y)$ with $y\neq 0$, with the following properties:
\begin{enumerate}
\item[i)]
When $x=0$ and $y\neq 0$, the equation \eqref{rectas8} is reduced to $dx dy=0$. This means that the line fields $\{\cF_a,\cF_b\}$ that solve \eqref{rectas8} are one tangent and the other normal to the segment $\Gamma$ at the points in $\Gamma\cap\Omega^*$. 
\item[ii)]
When $x\neq 0$, the \emph{cross field} given at each $(x,y)$ by $\cF_a\cup \cF_b$ coincides with the cross field determined by the principal lines given by \eqref{rectas1}, i.e., with $\cL_1\cup\cL_2$.
\end{enumerate}

Let $\cF_a$ denote the analytic line field that solves \eqref{rectas8} and is tangent to the $x$-axis (i.e., to $\Gamma)$ at $(0,y)$. Since the solutions to \eqref{rectas8} are eigenlines of the matrix $(\hat\alfa_{ij})$ in \eqref{eq:ago}, we see that $\cF_a$ is associated to the non-zero eigenvalue of $(\hat\alfa_{ij})$ at $(0,y)$, which has the same sign as $\eta(0,y)$.

Let us detect next if, for $(x,y)$ close to $(0,y)$ with $x\neq 0$, we have $\cF_a=\cL_1$ or $\cF_a=\cL_2$ at $(x,y)$. For this, let us recall that $\kappa_1\geq \kappa_2$ by convention, so $\cL_1$ (resp. $\cL_2$) is the principal line field associated to the non-negative (resp. non-positive) principal curvature $\kappa_1$ (resp. $\kappa_2)$. By our previous discussion on $\cF_a$, we have:

\begin{enumerate}
\item[(a)]
If $n$ is even, then at any $(x,y)$ sufficiently close to $(0,y)$ it holds $\cF_a=\cL_1$ if $\eta(0,y)>0$ and $\cF_a = \cL_2$ if $\eta(0,y)<0$. In particular, each of the principal line fields $\cL_1,\cL_2$ extends analytically across $\Gamma$.
\item[(b)]
If $n$ is odd, then at any $(x,y)$ sufficiently close to $(0,y)$ it holds $\cF_a=\cL_1$ if $x \eta(0,y)>0$ and $\cF_a = \cL_2$ if $x\eta(0,y)<0$. In particular, $\cL_1,\cL_2$ do not extend analytically across $\Gamma$, and $\cF_a$ changes from $\cL_1$ to $\cL_2$ (or vice versa) as we cross $\Gamma$.
\end{enumerate}

Note that it follows from properties i) and ii) above that, making $\Omega$ smaller if necessary, the cross field of principal lines $\cL_1\cup \cL_2$, at first only defined on $\Omega^*\setminus \cup_j \Gamma_j$, can be analytically extended to a cross field $\cK$ in $\Omega^*$. We can \emph{locally} write $\cK=\cF_1\cup \cF_2$ where $\cF_1,\cF_2$ are real analytic line fields. We want to show next that $\cK$ can also be \emph{globally} split as the union of two real analytic line fields in $\Omega^*$, i.e., we seek to prove that these local line fields $\cF_1,\cF_2$ can be analytically extended to $\Omega^*$. For that, we need to rule out the possibility depicted in Figure \ref{fig:cruces}. That is, since $\cK$ is well defined on $\Omega^*$, we need to check that if $C$ is a small circle in $\Omega^*$ enclosing the origin and parametrized by $\alfa:[0,1]\flecha\Omega^*$, with $\alfa(0)=\alfa(1)$, then the analytic continuation along $\alfa(t)$ of the line field $\cF_1$ at $\alfa(0)$ ends up being $\cF_1$ (and not $\cF_2$) when $t=1$. 

\begin{figure}[htbp]
        \includegraphics[width=11cm]{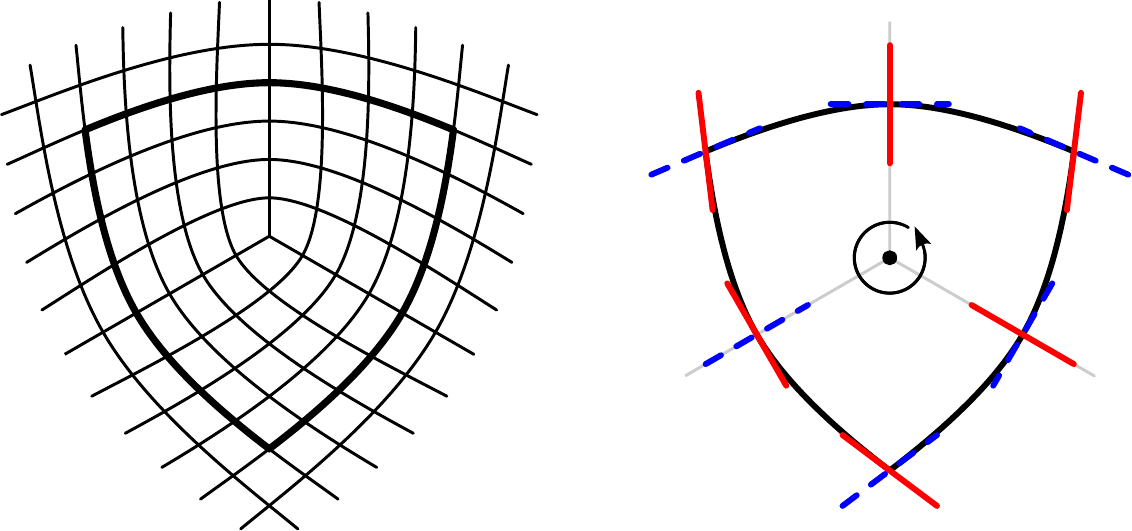} 
     \caption{A well-defined analytic cross field $\cK=\cF_1\cup \cF_2$, with an isolated singularity at the origin, that does not split into two global analytic line fields. As we travel around the origin following the highlighted curve, starting from the upper right corner, the analytic continuation of the \emph{red} local line field of $\cK$ becomes the \emph{blue} one after one turn.} 
\label{fig:cruces}
\end{figure}

For this, assume without loss of generality that $p_0=\alfa(0)=\alfa(1)\in \Omega^*\setminus \cup_j \Gamma_j$, and let $\cF_1$ be one of the local line fields for which $\cK=\cF_1\cup \cF_2$ holds around $p_0$. For any $t\in [0,1]$, we denote by $\cF(t)$ the analytic continuation of $\cF_1$ along $\alfa(t)$, and want to show that $\cF(0)=\cF(1)$. First of all, note that the umbilical segments $\Gamma_1,\dots, \Gamma_r$ divide $\Omega^*$ into $2r$ circular open sectors $D_1,\dots, D_{2r}$. By item ii) above, on each circular arc $C\cap D_{j}$ we have either $\cF(t)=\cL_1(\alfa(t))$ or $\cF(t)=\cL_2(\alfa(t))$. 
Moreover, as $\alfa(t)$ crosses an umbilic segment $\Gamma$ separating $D_j$ from $D_{j+1}$, we have two possibilities:

\begin{enumerate}
\item
If the number $n>2$ associated to $\Gamma$ by \eqref{achen} is \emph{even}, then it follows from item (a) above that we have one of either $\cF(t)=\cL_1(\alfa(t))$ or $\cF(t)=\cL_2(\alfa(t))$ at $D_j\cup D_{j+1}$.
\item
If $n$ is \emph{odd}, then by item (b) above we have $\cF(t)=\cL_1(\alfa(t))$ in $D_j$ and $\cF(t)=\cL_2(\alfa(t))$ in $D_{j+1}$, or vice versa.
\end{enumerate}

Since the circle $C$ intersects each segment $\Gamma$ exactly twice, we deduce from items (1), (2) above that if $\cF(0)=\cL_i(p_0)$ for some $i\in \{1,2\}$, then $\cF(1)=\cL_i(p_0)$ for the same $i$. This proves the desired equality $\cF(0)=\cF(1)$ and prevents the situation in Figure \ref{fig:cruces}.

As a consequence, $\cK=\cF_1\cup \cF_2$, where $\cF_1,\cF_2$ are analytic line fields in $\Omega^*$. Moreover, for every $p\in \Gamma_j\cap \Omega^*$, one of $\cF_1(p),\cF_2(p)$ is tangent to $\Gamma_j$ at $p$, and the other one is therefore normal. This proves Assertion \ref{ass:1}.
\end{proof}

We are going to prove next that the index at $(0,0)$ of the analytic line fields $\cF_1,\cF_2$ defined by Assertion \ref{ass:1} is $\leq 0$. 

To start, let $\nu\in \S^1$ be a direction for which Theorem \ref{th:analytic} holds. Up to a rotation in the $(x,y)$-coordinates, and taking into account Remark \ref{rem:nu}, we can assume that $\nu=(0,1)$, i.e., that 
\begin{equation}\label{nui}
h_{\nu} =h_y,
\end{equation} 
and that none of the umbilical segments $\Gamma_j$ is horizontal, i.e., all of them are transverse to $(1,0)$.

For these $(x,y)$-coordinates we define next a matrix homotopy 
\begin{equation}\label{homoto}
(m_{ij}(t))= (1-t) (\alfa_{ij}) + t (h_{ij}),
\end{equation} 
where $(h_{ij})=D^2h$, and consider for each $t\in [0,1]$ the equation
 \begin{equation}\label{rectas3}
 -m_{12}(t) dx^2 + (m_{11}(t)-m_{22}(t)) dx dy + m_{21}(t) dy^2=0.
 \end{equation}
Note that \eqref{rectas3} coincides with \eqref{rectas1} for $t=0$, while for $t=1$ is given by
\begin{equation}\label{rectas2}
-h_{xy}(dx^2 -dy^2) + (h_{xx}-h_{yy}) dx dy=0.
\end{equation} 
We are going to show in Assertion \ref{ass:2} below that the solutions to \eqref{rectas3} define, for each $t\in [0,1]$, a pair of real analytic line fields $\cF_1^t,\cF_2^t$ in $\Omega^*$. This result is in analogy with Assertion \ref{ass:1}, that actually corresponds to Assertion \ref{ass:2} in the particular case $t=0$. Once we prove this, such line fields will have an associated index at the origin that, by topological invariance of the index, must be the same for every $t\in [0,1]$. Subsequently, we will prove in Assertion \ref{ass:3} that the index of the solutions to \eqref{rectas2}, i.e. of \eqref{rectas3} for $t=1$, is $\leq 0$. From there, we obtain the desired result that the index at the origin of the line fields $\cF_1=\cF_1^0$ and $\cF_2=\cF_2^0$ is $\leq 0$.

We start with:

\begin{assertion}\label{ass:2}
For each $t\in [0,1]$, there exist two real analytic line fields $\cF_1^t,\cF_2^t$ in $\Omega^*$ that solve \eqref{rectas3} at every point of $\Omega^*$. Moreover, for every $p\in \Gamma_j\cap \Omega^*$, one of $\cF_1^t(p),\cF_2^t(p)$ is tangent to $\Gamma_j$ at $p$, and the other one is normal.

\end{assertion}
\begin{proof}
By its own construction, and taking into account \eqref{2ff}, the matrix $(m_{ij}(t))$ is of the form 
\begin{equation}\label{matm}
(m_{ij}(t))=(h_{ij})\cdot P(t),
\end{equation} 
where $P(t)$ is positive definite. Equation \eqref{rectas3} defines a pair of straight lines passing through the origin except at the points $p\in \Omega$ where 
\begin{equation}\label{detra}
{\rm det}(m_{ij}(t)) =0={\rm trace}(m_{ij}(t)).
\end{equation} 
Let $\landa_1,\landa_2$ be the eigenvalues of $(h_{ij})$, $Q$ a matrix that diagonalizes $(h_{ij})$, and $\Delta:=Q P(t) Q^{-1}$. Since $P(t)$ is positive definite, we have from \eqref{matm} that 
\begin{equation}\label{detra2}
 {\rm trace}(m_{ij}(t)) ={\rm trace}(Q (h_{ij})Q^{-1} Q P(t) Q^{-1})= \landa_1 \Delta_{11} + \landa_2 \Delta_{22},
 \end{equation} 
with $\Delta_{11}>0, \Delta_{22}>0$. In particular, if ${\rm det}(m_{ij}(t)) =0$, we have $\landa_1 \landa_2=0$ by \eqref{matm}, and so from \eqref{detra2}, ${\rm trace}(m_{ij}(t)) = \delta^2\,  {\rm trace}(h_{ij})$ for some $\delta>0$. This shows that \eqref{detra} can only happen at some $p\in \Omega$ if $(h_{ij})$ vanishes at $p$. Therefore, we conclude that the solutions to \eqref{rectas3} define a real analytic cross field on $\Omega^*\setminus \cup_j \Gamma_j$.

Along the umbilical segments $\Gamma_j$, the same analytic continuation argument used in the proof of Assertion \ref{ass:1} still works in this case. For instance, assume that $\Gamma_j$ is contained in the line $x_{\theta}=0$, where we are denoting $$x_{\theta}:=\cos \theta x+\sin\theta y, \hspace{0.5cm} y_{\theta}:= -\sin \theta x + \cos \theta y.$$ Then, similarly to \eqref{achen}, we have $h(x,y)=x_{\theta}^n \,\eta(x_{\theta},y_{\theta})$ for some $n>2$, and the equation corresponding to \eqref{eq:ago} in this more general context is, by \eqref{homoto} and \eqref{eq:ago},
$$(m_{ij}(t))= x_{\theta}^{n-2}\left[ \left(\begin{array}{cc}n(n-1)\eta(0,y_{\theta})\left((1-t)(1+y_{\theta}^2)^{-2} +t\right)& 0\\ 0 & 0 \end{array} \right) + x_{\theta} \left(\begin{array}{cc} * & *\\ * & * \end{array} \right)\right].$$
From here, and arguing as in Assertion \ref{ass:1} using this time \eqref{rectas3}, we can deduce that, for each $t\in [0,1]$, the solutions to \eqref{rectas3} can be extended to define an analytic cross field $\cK^t$ on $\Omega^*$. Since we had previously shown that the cross field $\cK=\cK^0$ can be described as $\cK=\cF_1\cup \cF_2$ where both $\cF_i$ are analytic line fields on $\Omega^*$, it follows by a continuity argument that for all $t\in [0,1]$ we have $\cK^t =\cF_1^t\cup \cF_2^t$, for two analytic line fields $\cF_1^t,\cF_2^t$ in $\Omega^*$. Note that $\cF_i^0 =\cF_i$, where $\cF_1,\cF_2$ are the line fields constructed above. Moreover, similarly to item i) in Assertion \ref{ass:1}, at each $p\in \Gamma_j\cap\Omega^*$ one of these line fields $\cF_i^t(p)$ is tangent to $\Gamma_j$ and the other one is orthogonal to $\Gamma_j$. In particular, this holds for \eqref{rectas2}, making $t=1$.
\end{proof}

Assertion \ref{ass:2} shows that all the analytic line fields $\cF_i^t$ have a well defined topological index at the origin. We next apply Theorem \ref{th:analytic} to control this index for $t=1$.
\begin{assertion}\label{ass:3}
The topological index of the line fields $\cF_1^1,\cF_2^1$ at the origin is $\leq 0$
\end{assertion}
\begin{proof}
We start by considering the vector field 
\begin{equation}\label{defz}
Z=\left(-2h_{xy},h_{xx}-h_{yy}\right),
\end{equation} 
which has no zeros on $\Omega^*\setminus \cup_j \Gamma_j$ since $h_{xx}h_{yy}-h_{xy}^2\leq 0$. In order to understand the behavior of $Z$ along an umbilic segment $\Gamma_j$, we assume that $\Gamma_j$ is contained in a line $\cos \theta x + \sin \theta y=0$, and so $h$ can be written as $$h(x,y)=(\cos \theta x+\sin \theta y)^n \, \eta(x,y)$$ for some $n>2$, with $\eta(x,y)\neq 0$ for $(x,y)\in \Gamma_j\cap \Omega^*$. Then, arguing as in Lemma \ref{lem:ext}, we obtain a similar formula to \eqref{fraca} in this context:
\begin{equation}\label{fraka2}
\frac{Z(x,y)}{|Z(x,y)|} \rightarrow \pm (-\sin (2\theta),\cos(2\theta))
\end{equation}
as $(x,y)$ approaches a point $p\in \Gamma_j\cap \Omega^*$ through a trajectory that lies on one side of $\Gamma_j$. Here, as in \eqref{fraca}, the $\pm$ sign depends on the parity of $n$, the sign of $\eta(p)$ and the sign of $\cos \theta x + \sin \theta y $ along the chosen trajectory. 

Let now $\cL_Z$ denote the line field generated by $Z$. While $\cL_Z$ is at first only defined when $Z\neq 0$, i.e., in $\Omega^*\setminus \cup_j \Gamma_j$, the behavior of $Z$ at each $\Gamma_j$ given by \eqref{fraka2} implies that $\cL_Z$ can be extended to a real analytic line field in $\Omega^*$, that we still denote by $\cL_Z$. Let ${\rm Ind}(\cL_z)$ denote the topological index of $\cL_Z$ at the origin. We next relate ${\rm Ind}(\cL_Z)$ with the topological index of $\cF_1^1$ and $\cF_2^1$.

Let $\nu=(\cos \tau,\sin\tau)\equiv (dx,dy)$ be a solution to \eqref{rectas2} at a point $p\in \Omega^*\setminus \cup_j \Gamma_j$, and let $(\cos\theta,\sin\theta)$ denote $Z(p)/|Z(p)|$. Then, \eqref{rectas2} can be rewritten using \eqref{defz} as 
\begin{equation}\label{recan}
\cos(\phi + 2\tau) =0.
\end{equation} 
By the analyticity of $\cL_Z$ and $\cF_i^1$ in $\Omega^*$ we deduce then from \eqref{recan} that $${\rm Ind}(\cL_Z)= 2\, {\rm Ind}(\cF_1^1) = 2\,{\rm Ind}(\cF_2^1).$$

Next, recall that for our choice of $(x,y)$-coordinates, the vector field $\nabla h_y$ is in the conditions of Theorem \ref{th:analytic}, see \eqref{nui}. In particular, by Theorem \ref{th:analytic}, $\nabla h_y$ generates a line field $\cF$ in $\Omega^*$ that has non-positive topological index at the origin. We next compare $Z$ with $\nabla h_y$. First, by $h_{xx} h_{yy}-h_{xy}^2\leq 0$ we have $$\esiz Z,\nabla h_y\esde \leq -|\nabla h_y|^2 \leq 0.$$ By Theorem \ref{th:analytic}, the set of critical points of $h_y$ coincides with $\cup_j \Gamma_j$, so $\esiz Z,\nabla h_y\esde<0$ on $\Omega^*\setminus \cup_j \Gamma_j$. That is, $\cF$ and $\cL_Z$ are never orthogonal in $\Omega^*\setminus \cup_j \Gamma_j$.

Consider now some umbilical segment $\Gamma_j$, that we assumed contained in $\cos \theta x + \sin \theta y=0$, and take $p\in \Gamma_j\cap \Omega^*$. By Theorem \ref{th:analytic}, the line field $\cF(p)$ is orthogonal to $\Gamma_j$ at $p$, and by \eqref{fraka2} the line field $\cL_Z(p)$ points in the direction $(-\sin(2\theta),\cos(2\theta))$ at $p$. Since $\Gamma_j$ is not horizontal (see our working conditions after \eqref{nui}), we deduce that $\cF$ and $\cL_Z$ are still not orthogonal at $p$. Thus, $\cF$ and $\cL_Z$ are never orthogonal in $\Omega^*$, and so they have the same rotation index at $(0,0)$ and we conclude that
$$ 2\, {\rm Ind}(\cF_1^1) = 2\,{\rm Ind}(\cF_2^1)= {\rm Ind}(\cL_Z)={\rm Ind}(\cF)\leq 0,$$where the last inequality comes from Theorem \ref{th:analytic}. This proves Assertion \ref{ass:3}.
\end{proof}

We can now finish the proof of Theorem \ref{th:indice}. By Assertion \ref{ass:2}, all the line fields $\cF_i^t$, $i=1,2$, have a well defined topological index at the origin. By the topological invariance of the index and Assertion \ref{ass:3}, we have ${\rm Ind}(\cF_i^t)\leq 0$ for all $t\in [0,1]$ and $i=1,2$. Since $\cF_i^0=\cF_i$, we conclude that ${\rm Ind}(\cF_i)\leq 0$. This completes the proof of Theorem \ref{th:indice}.
\end{proof}

\begin{corollary} {\bf (Theorem \ref{th:main1})}
Any real analytic sphere $\Sigma$ immersed in $\S^3$ with $\kappa_1 \kappa_2\leq 0$ at every point is totally geodesic.
\end{corollary}
\begin{proof}
Since $\kappa_1\kappa_2\leq 0$, the umbilical points of $\Sigma$ are those where $II=0$. Assume that $\Sigma$ is not totally geodesic. Then, by analyticity, its umbilic set $\cU$ is a finite union of isolated points and closed, regular embedded real analytic curves in $\Sigma$. These curves are actually geodesics of $\S^3$ contained in $\Sigma$, by Lemma \ref{lem:recta} and equation \eqref{2ff}. We denote them by $\Gamma_j$, as usual.

Let $\cK=\cL_1\cup \cL_2$ denote the \emph{cross field} on $\Sigma\setminus \cU$ defined by the two principal line fields $\cL_1,\cL_2$ of $\Sigma$. We proved in Theorem \ref{th:indice} that for any $p\in \Sigma$ there exists a punctured neighborhood $D^*(p)$ to which $\cK$ can be analytically extended as a cross field. By compactness of $\Sigma$, we deduce then that $\cK$ can be extended to an analytic cross field on $\Sigma\setminus \{q_1,\dots, q_s\}$, for a certain finite number of points $q_i\in \Sigma$.

Let now $p\in \Sigma\setminus \cU$ be an arbitrary non-umbilic point. Let $\alfa,\gamma$ be two trajectories from $p$ to some point $p'\in \Sigma\setminus \{q_1,\dots , q_s\}$, and let $\cF_1^{\alfa}$ (resp. $\cF_1^{\gamma}$) be the analytic continuation of $\cL_1(p)$ along $\alfa$ (resp. $\gamma$). We assume for simplicity that $\alfa,\gamma$ intersect transversely all geodesics $\Gamma_j$, and that $p'\not\in \cup_j \Gamma_j$. In these conditions, \emph{we claim that}
\begin{equation}\label{fag}
\cF_1^{\alfa}(p')=\cF_1^{\gamma}(p').
\end{equation}

Indeed, each of the totally geodesic circles $\Gamma_j$ separates $\Sigma$ into two connected components, since $\Sigma$ is diffeomorphic to $\S^2$. In particular, by transversality, the number of points in $\alfa\cap \Gamma_j$ is odd (resp. even) if and only if so is the number of points in $\gamma\cap \Gamma_j$. In addition, we proved in Theorem \ref{th:indice} that, for each $\Gamma_j$, one of the following two possibilities holds (see (a) and (b) in the proof of Theorem \ref{th:indice}):
\begin{enumerate}
\item[A)] $\cL_1$ and $\cL_2$ extend analytically across $\Gamma_j$. In particular, $\cF_1^{\alfa}$ is equal to one of $\cL_1$ or $\cL_2$ along a small arc of $\alfa$ around each $p\in \alfa\cap \Gamma_j$, and the same holds for $\gamma$ and $\cF_1^{\gamma}$.
\item[B)] $\cL_1$ and $\cL_2$ do not extend analytically across $\Gamma_j$. In that case, for any small arc of $\alfa$ containing $p\in \alfa\cap \Gamma_j$, the line field $\cF_1^{\alfa}$ is equal to $\cL_1$ on one side of $p$, and to $\cL_2$ on the other side of $p$. The same property holds for $\cF_1^{\gamma}$.
\end{enumerate}

In this way, it follows from A), B) and the equal parity of the sets $\alfa\cap \Gamma_j$ and $\gamma\cap \Gamma_j$ for every $\Gamma_j$ that \eqref{fag} holds.

The path independence property \eqref{fag} clearly implies that the local line fields $\cL_1,\cL_2$ around $p$ can be extended analytically to define respective global line fields $\cF_1,\cF_2$ on $\Sigma\setminus \{q_1,\dots , q_s\}$, with the property that $\cK=\cF_1\cup \cF_2$ at every point. These line fields $\cF_1,\cF_2$ obviously agree with the analytic line fields defined on Theorem \ref{th:indice} around each singularity $q_i$. So, by Theorem \ref{th:indice}, the topological index of both $\cF_1,\cF_2$ around each $q_i$ is $\leq 0$. This contradicts that, by Poincaré-Hopf, $\sum_{i=1}^s {\rm Ind}(\cF_1,q_i)={\rm Ind}(\cF_2,q_i)=2$, since $\Sigma$ is diffeomorphic to $\S^2$. This contradiction shows that $\Sigma$ must be totally geodesic, and completes the proof.
\end{proof}

\def\refname{References}

\vskip 0.2cm

\noindent José A. Gálvez

\noindent Departamento de Geometría y Topología,\\ Instituto de Matemáticas IMAG,\\
Universidad de Granada (Spain).

\noindent  e-mail: {\tt jagalvez@ugr.es}

\vskip 0.2cm

\noindent Pablo Mira

\noindent Departamento de Matemática Aplicada y Estadística,\\ Universidad Politécnica de Cartagena (Spain).

\noindent  e-mail: {\tt pablo.mira@upct.es}

\noindent Marcos Paulo Tassi

\noindent Dipartimento di Ingegneria e Scienze dell'Informazione e Matematica, \\ Università degli Studi dell'Aquila (Italy)

\noindent  e-mail: {\tt marcospaulo.tassi@univaq.it}

\noindent This research has been financially supported by: Projects PID2020-118137GB-I00 and CEX2020-001105-M, funded by MCIN/AEI /10.13039/501100011033; Junta de Andalucia grant no. P18-FR-4049; CARM, Programa Regional de Fomento de la Investigación, Fundación Séneca-Agencia de Ciencia y Tecnología Región de Murcia, reference 21937/PI/22; and Projects 2020/03431-6 and 2021/10181-9, funded by São Paulo Research Foundation (FAPESP).

\end{document}